\documentclass{amsart}
\usepackage{amssymb}
\usepackage{graphicx}
\numberwithin{equation}{section}

\theoremstyle{plain} % This is the default.
\newtheorem{thm}[equation]{Theorem}
\newtheorem{cor}[equation]{Corollary}
\newtheorem{lem}[equation]{Lemma}
\newtheorem{prop}[equation]{Proposition}
\newtheorem{claim}{Claim}

\theoremstyle{definition}
\newtheorem{defn}[equation]{Definition}

\theoremstyle{remark}
\newtheorem{rem}[equation]{Remark}

\title[Pairs of disjoint primitives]{A classification of pairs of disjoint nonparallel primitives in the boundary of a genus two handlebody}
\author{John Berge}
\date{\today}                                          

\begin{document}

\begin{abstract}
Embeddings of pairs of disjoint nonparallel primitive simple closed curves in the boundary of a genus two handlebody are classified. Briefly, two disjoint primitives either lie on opposite ends of a product $F \boldsymbol{\times} I$, or they lie on opposite ends of a kind of ``twisted'' product $F \widetilde{\boldsymbol{\times}} I$, where $F$ is a once-punctured torus. 

If one of the curves is a proper power of a primitive, the situation is simpler. Either the curves lie on opposite sides of a separating disk in the handlebody, or they bound a nonseparating essential annulus in the handlebody. 
\end{abstract}

\maketitle

\section{Introduction}

Suppose $H$ is a genus two handlebody. A simple closed curve $\alpha$ in $\partial H$ is \emph{primitive} in $H$ if there exists a disk $D$ in $H$ such that $|\alpha \cap D| = 1$. Equivalently $\alpha$ is conjugate to a free generator of $\pi_1(H)$. A pair of disjoint properly embedded disks in $H$ is a \emph{complete set of cutting disks} of $H$ if cutting $H$ open along the pair of disks yields a 3-ball.
A pair of disjoint simple closed curves $(\alpha, \beta)$ in $\partial H$ is a \emph{primitive pair} if there is a complete set of cutting disks $\{D_A, D_B\}$ of $H$ such that $|\alpha \cap \partial D_A| = 1$, $|\alpha \cap \partial D_B| = 0$, $|\beta \cap \partial D_A| = 0$, and $|\beta \cap \partial D_B| = 1$. Equivalently, the pair $(\alpha, \beta)$ is conjugate to a pair of free generators of $\pi_1(H)$. A pair of disjoint nonparallel simple closed curves $(\alpha, \beta)$ in $\partial H$ is a \emph{pair of primitives} if both $\alpha$ and $\beta$ are primitive in $H$. (Note that a ``pair of primitives'' is not generally a ``primitive pair''.) A pair of nonseparating simple closed curves $(\alpha, \beta)$ in $\partial H$ is \emph{separated} in $H$ if there exists a separating disk $D$ embedded in $H$ such that $\alpha$ and $\beta$ lie on opposite sides of $\partial D$ in $\partial H$.

\section{Preliminaries}

This section recalls some of the basic properties of genus two Heegaard diagrams, their underlying graphs, and genus two R-R diagrams which will be helpful.

\subsection{Genus two Heegaard diagrams and their underlying graphs} \label{H-diagrams and underlying graphs}

Suppose $\alpha$ and $\beta$ are disjoint nonparallel simple closed curves in the boundary of a genus two handlebody $H$, neither $\alpha$ nor $\beta$ bound disks in $H$, and $\{D_A, D_B\}$ is a complete set of cutting disks of $H$. Cutting $H$ open along $D_A$ and $D_B$ cuts $\alpha$ and $\beta$ into sets of arcs $E(\alpha)$ and $E(\beta)$ respectively, and cuts $H$ into a 3-ball $W$. Then $\partial W$ contains disks $D_A^+$, $D_A^-$, $D_B^+$, and $D_B^-$ such that gluing $D_A^+$ to $D_A^-$ and $D_B^+$ to $D_B^-$ reconstitutes $\alpha$, $\beta$, and $H$. 

The sets of arcs $E(\alpha)$ and $E(\beta)$ form the edges of \emph{Heegaard diagrams} in $\partial W$ with ``fat'', i.e. disk rather than point, vertices $D_A^+$, $D_A^-$, $D_B^+$, and $D_B^-$. Let $HD_\alpha$, $HD_\beta$, and $HD_{\alpha, \beta}$ be the Heegaard diagrams in $\partial W$ whose edges are the arcs of $E(\alpha)$, $E(\beta)$, and $E(\alpha) \cup E(\beta)$ respectively. 

If one ignores how $D_A^+$ and $D_B^+$ are identified with $D_A^-$ and $D_B^-$ to reconstitute $H$, the sets of arcs $E(\alpha)$, $E(\beta)$, and $E(\alpha) \cup E(\beta)$ also form the edges of graphs in $\partial W$ with vertices $D_A^+$, $D_A^-$, $D_B^+$, and $D_B^-$. Let $G_\alpha$, $G_\beta$, and $G_{\alpha, \beta}$ denote the graphs in $\partial W$ whose edges are the arcs of $E(\alpha)$, $E(\beta)$, and $E(\alpha) \cup E(\beta)$ respectively. Then $G_\alpha$ is the graph \emph{underlying} $HD_\alpha$, and $G_\beta$ is the graph \emph{underlying} $HD_\beta$ etc..  Note that 
these graphs are not just abstract graphs, since they inherit specific embeddings in the 2-sphere $S^2 \cong \partial W$ from the Heegaard diagrams which they underlie.

\begin{rem}
The notation $HD_\alpha$, $HD_\beta$, and $HD_{\alpha, \beta}$, does not specify the set of cutting disks $\{D_A, D_B\}$. However, this shouldn't lead to ambiguity because context will make it clear which set of cutting disks of $H$ is playing the role of $\{D_A, D_B\}$.
\end{rem}

\begin{rem}
In figures of $G_\alpha$, $G_\beta$, or $G_{\alpha, \beta}$,  the disks $D_A^+$, $D_A^-$, $D_B^+$, and $D_B^-$ in $\partial W$ are smashed to points denoted by $A^+$, $A^-$, $B^+$, and $B^-$ respectively.
\end{rem}

\subsection{Genus two R-R diagrams} \label{Genus two R-R diagrams}

R-R diagrams are a type of planar diagram related to Heegaard diagrams. These diagrams were originally introduced by Osborne and Stevens in \cite{OS74}. They are particularly useful for describing embeddings of simple closed curves in the boundary of a handlebody so that the embedded curves represent certain conjugacy classes in $\pi_1$ of the handlebody. 

Here is a description of the basics of genus two R-R diagrams, which is all we need. 
Suppose $\Sigma$ is a closed orientable surface of genus two obtained by capping off the two boundary components of an annulus $\mathcal{A}$ with a pair of once-punctured tori $F_A$ and $F_B$, so that $\Sigma = \mathcal{A} \cup F_A \cup F_B$, $\mathcal{A} \cap F_A = \partial F_A$, and $\mathcal{A} \cap F_B = \partial F_B$. Following Zieschang \cite{Z88}, the separating simple closed curves $\partial F_A$ and $\partial F_B$ in $\Sigma$ are \emph{belt} curves, and $F_A$ and $F_B$ are \emph{handles}.

If $\mathcal{S}$ is a set of pairwise disjoint simple closed curves in $\Sigma$, then, after isotopy, we may assume each curve $\zeta \in \mathcal{S}$ is either disjoint from $\partial F_A \cup \partial F_B$, or $\zeta$ is cut by its intersections with $\partial F_A \cup F_B$ into arcs, each properly embedded and essential in one of $\mathcal{A}$, $F_A$, $F_B$. A properly embedded essential arc in $F_A$ or $F_B$ is a \emph{connection}. Two connections in $F_A$ or $F_B$ are \emph{parallel} if they are isotopic in $F_A$ or $F_B$ via an isotopy keeping their endpoints in $\partial F_A$ or $\partial F_B$. A collection of pairwise disjoint connections on a given handle can be partitioned into \emph{bands} of pairwise parallel connections. Since each handle is a once-punctured torus, there can be at most three nonparallel bands of connections on a given handle.

Note that sets of pairwise nonparallel connections in a once-punctured torus are unique up to homeomorphism. In particular, if $F$ and $F'$ are once-punctured tori, $\Delta = \{\delta_1, \cdots ,\delta_i \}$, is a set of pairwise nonparallel connections in $F$, and $\Delta' = \{\delta'_1, \cdots ,\delta'_i \}$, $1 \leq i \leq 3$ is a set of pairwise nonparallel connections in $F'$, then there is a homeomorphism $h \colon F \to F'$ which takes $\Delta$ to $\Delta'$. 

\begin{rem}
In practice, it is often inconvenient to have curves in $\mathcal{S}$ that lie completely in $F_A$ or $F_B$. This situation can be avoided by relaxing the supposition that each curve in $\mathcal{S}$ has only essential intersections with the belt curves $\partial F_A$ and $\partial F_B$. Then, if a curve $\zeta \in \mathcal{S}$ lies completely in $F_A$ or $F_B$, say $F_B$, $\zeta$ can be isotoped in $\Sigma$ so that $\zeta \cap F_B$ consists of one properly embedded essential arc, while $\zeta \cap \mathcal{A}$ is an inessential arc in $\mathcal{A}$, isotopic in $\mathcal{A}$ into $\partial F_B$, keeping its endpoints fixed.
(In the R-R diagrams of this paper, the curve $\beta$ is always displayed in this manner.)
\end{rem}

Some simplifications can be made at this point without losing any information about the embedding of the curves of $\mathcal{S}$ in $\Sigma$. For example, suppose $F$ is either $F_A$ or $F_B$, and let $\mathcal{S}_\mathcal{A}$ be the set of arcs in which curves of $\mathcal{S}$ intersect $\mathcal{A}$. Then each set of parallel connections on $F$ can be merged into a single connection. (This also merges some endpoints of arcs in $\mathcal{S}_\mathcal{A}$ meeting $\partial F$.) 

After such mergers, $F$ carries at most $3$ pairwise nonparallel connections.  Continuing, after each set of parallel connections on $F_A$ and $F_B$ has been merged, additional mergers of sets of properly embedded parallel subarcs of $\mathcal{S}_\mathcal{A}$ can also be made; although now whenever, say, $n$ parallel arcs are merged into one, this needs to be recorded by placing the integer $n$ near the single arc resulting from the merger.

Merging parallel connections in $F_A$ and $F_B$ turns the set of pairwise disjoint simple closed curves in $\mathcal{S}$ into a graph $\mathcal{G}$ in $\Sigma$ whose vertices are the endpoints of the remaining connections in $F_A$ and $F_B$. Clearly $\mathcal{G}$ and its embedding in $\Sigma$ completely encodes the embedding of the curves of $\mathcal{S}$ in $\Sigma$. 

Understanding $\mathcal{G}$ and the curves of $\mathcal{S}$ it represents is complicated by the fact that $\mathcal{G}$ is usually nonplanar. However, $\mathcal{G}$ always has a nice immersion $I(\mathcal{G})$ in the plane $\mathbb{R}^2$, which we now describe. It is this immersion of $\mathcal{G}$ in $\mathbb{R}^2$ which becomes an R-R diagram of the curves of $\mathcal{S}$ in $\Sigma$.

To produce $I(\mathcal{G})$, first remove a small disk $D$, disjoint from $\mathcal{G}$, from the interior of $\mathcal{A}$. Then embed $\mathcal{A} - D$  in $\mathbb{R}^2$ so that $\partial F_A$ and $\partial F_B$ bound disjoint round disks, say $\mathcal{F}_A$ and $\mathcal{F}_B$ respectively, in $\mathbb{R}^2$. 

Next, note that if $\delta$ and $\delta'$ are nonparallel connections on a handle $F_X$, with $X \in \{A,B\}$, the endpoints of $\delta$ separate the endpoints of $\delta'$ in the belt curve $\partial F_X$. It follows that, if $u$ and $v$ in $\partial F_X$ are the endpoints of a connection $\delta$ in $F_X$, we may assume $u$ and $v$ bound a diameter $d_\delta$ of the disk $\mathcal{F}_X$, and then $\delta$ can be embedded in $\mathcal{F}_X$ as the diameter $d_\delta$ of $\mathcal{F}_X$. This results in each round disk $\mathcal{F}_X$ containing $0$, $1$, $2$, or $3$ diameters passing through its center $X$, with each diameter an image of a connection in $F_X$, where the number of such diameters depends upon whether $F_X$ originally contained respectively $0$, $1$, $2$, or $3$ bands of parallel connections.

Also note that the immersion $I(\mathcal{G})$ still encodes the embedding of the curves in $\mathcal{S}$ in $\Sigma$ up to homeomorphism. This follows from the aforementioned fact that if $\Delta = \{\delta_1, \cdots ,\delta_i \}$ and $\Delta' = \{\delta'_1, \cdots ,\delta'_i \}$, $1 \leq i \leq 3$ are each sets of pairwise nonparallel connections in a once-punctured torus $F$, then there is a homeomorphism of $F$ which takes $\Delta$ to $\Delta'$.

The partition of $\Sigma$ into $\mathcal{A}$, $F_A$ and $F_B$ makes it easy to describe infinite families of parametrized embeddings of the curves of $\mathcal{S}$ in the boundary of a genus two handlebody. To do this, consider $\Sigma$ as 2-sided, with sides $\Sigma^+$ and $\Sigma^-$, and suppose the curves of $\mathcal{S}$ lie in $\Sigma^+$. Then gluing a pair of disks $D_A$ and $D_B$ to $\Sigma^-$ so that $\partial D_A$ is glued to a nonseparating simple closed curve in $F_A^-$, $\partial D_B$ is glued to a nonseparating simple closed curve in $F_B^-$, and the resulting 2-sphere boundary component is capped off with a 3-ball, makes $\Sigma$ the boundary of a genus two handlebody $H$.

Continuing, note that if $\delta$ and $\delta'$ are two nonparallel oriented connections in a once-punctured torus $F$, then the isotopy class of an oriented nonseparating simple closed curve $\gamma$ in $F$ is determined by its algebraic intersection numbers with $\delta$ and $\delta'$. This makes it possible to parametrize the isotopy classes of attaching curves of $\partial D_A$ and $\partial D_B$ in $F_A$ and $F_B$ by adding integer labels to the endpoints of the diameters of $\mathcal{F}_A$ and $\mathcal{F}_B$ which represent connections in $F_A$ and $F_B$.  (There are minor restrictions on the values of these parameters; all related to similar restrictions on meridional and longitudinal coordinates of simple closed curves in $H_1(\partial V)$, where $V$ is a solid torus.  Figure \ref{DisPCFig6a} illustrates these.)

\begin{figure}[htbp]
\centering
\includegraphics[width = 1.0 \textwidth]{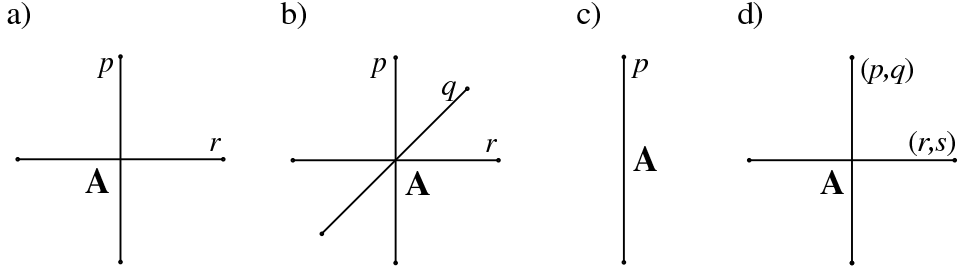}
\caption{There is a simple closed curve in the once-punctured torus $F_A$ which intersects the bands of connections in Figures \ref{DisPCFig6a}a, \ref{DisPCFig6a}b, and \ref{DisPCFig6a}c with the indicated intersection numbers if and only if: A) $\gcd(p,r) = 1$ in Figure \ref{DisPCFig6a}a. B) $\gcd(p,r) = 1$ and $q = p + r$ in Figure \ref{DisPCFig6a}b. C) $p \in \mathbb{Z}$ in Figure \ref{DisPCFig6a}c.  Figure \ref{DisPCFig6a}d shows a variant labeling, useful when we wish to think of $F_A$ as carrying a meridional and longitudinal pair of simple closed curves, say $m$ and $l$, meeting transversely at a single point. In this case, if $\delta$ is a connection with label $(p,q)$ in $F_A$, so $[\delta] = p[l] + q[m]$ in $H_1(F_A,\partial F_A)$, then $ps-rq = \pm 1$.}
\label{DisPCFig6a}
\end{figure}

\begin{rem}
In practice, given a set of curves $\mathcal{S}$ in the boundary of a genus two handlebody $H$, we usually reverse the order in which the partition of $\partial H$ into $\mathcal{A}$, $F_A$, and $F_B$ and a complete set of cutting disks $\{D_A, D_B\}$ are chosen by choosing a desired set of cutting disks $\{D_A, D_B\}$ first, and then choosing an appropriate compatible partition of $\partial H$ into $\mathcal{A}$, $F_A$, and $F_B$ with $\partial D_A \subset F_A$ and $\partial D_B \subset F_B$.
\end{rem}

\section{The classification}

\begin{thm} \label{3 case theorem}
Suppose $H$ is a genus two handlebody, and $\alpha$ and $\beta$ are a pair of disjoint nonparallel simple closed curves in $\partial H$ such that both $\alpha$ and $\beta$ are primitive in $H$. Then $\alpha$ and $\beta$ have an R-R diagram with the form of Figure \emph{\ref{DisPCFig1a}}, \emph{\ref{DisPCFig2a}} or \emph{\ref{DisPCFig3a}}.
\end{thm}

\begin{proof}
Recall that if $\gamma$ is a primitive simple closed curve in the boundary of a genus two handlebody $H$, then there is a unique cutting disk of $H$ (up to isotopy) disjoint from $\gamma$, while there are an infinite number of cutting disks of $H$ that intersect $\gamma$ transversely exactly once. Then given $(\alpha,\beta)$, let $\{D_A, D_B\}$ be a complete set of cutting disks  of $H$ such that $|\beta \cap \partial D_A| = 0$, and such that $\partial D_B$ intersects $\alpha$ minimally subject to $|\beta \cap \partial D_B | = 1$. 
Now the goal is to show that if $|\alpha \cap \partial D_B| = 0$, $|\alpha \cap \partial D_B| = 1$, or $|\alpha \cap \partial D_B| > 1$, then $\alpha$ and $\beta$ have an R-R diagram with the form of Figure \ref{DisPCFig1a}, \ref{DisPCFig2a} or \ref{DisPCFig3a} respectively.

Consider the first possibility $|\alpha \cap \partial D_B| = 0$. In this case, since $\alpha$ is primitive in $H$, we must have $|\alpha \cap \partial D_A| = 1$, and then the pair of disjoint primitives $(\alpha, \beta)$ is also a pair of primitives in $\pi_1(H)$. It follows that the pair $(\alpha, \beta)$ has an R-R diagram with the form of Figure \ref{DisPCFig1a}.

Turning to the second possibility, suppose $|\alpha \cap \partial D_B| = 1$. In this case, let $C$ be a separating simple closed curve in $\partial H$, disjoint from $\partial D_A$, $\partial D_B$, and $\beta$, such that $C$ separates $\partial D_A$ and $\partial D_B$. Then let $\mathcal{A}$ be a small annular regular neighborhood of $C$ in $\partial H$, chosen so that $\mathcal{A}$ is also disjoint from $\partial D_A$, $\partial D_B$, and $\beta$. Then $\partial H - int(\mathcal{A})$ is the union of two once-punctured tori $F_A$ and $F_B$, with $\partial D_A \subset F_A$ and $\partial D_B \subset F_B$, and we may suppose that $\alpha$ has only essential intersections with $\mathcal{A}$, $F_A$ and $F_B$, as well as $\partial D_A$ and $\partial D_B$.
This partition of $\partial H$ into $\mathcal{A} \cup F_A \cup F_B$ provides a natural  framework of the sort described in Subsection \ref{Genus two R-R diagrams}, which leads to an R-R diagram describing how $\alpha$, $\beta$, $\partial D_A$, and $\partial D_B$, are configured in $\partial H$. 

In this case, this is quite easy. Since $\alpha$ and $\beta$ are disjoint, while $|\alpha \cap \partial D_B| = 1$, $\alpha \cap F_B$ must consist of a single connection. Then $\alpha \cap F_A$ must also consist of exactly one connection, which can be any connection in $F_A$. It follows that the pair $(\alpha, \beta)$ has an R-R diagram with the form of Figure \ref{DisPCFig2a}.

Finally, consider the last possibility $|\alpha \cap \partial D_B| > 1$. In this case, Lemma \ref{Case 3 Heegaard diagram} lays the foundation for the analysis by showing that the graph $G_{\alpha, \beta}$ underlying the  Heegaard diagram $HD_{\alpha, \beta}$ has the form of Figure \ref{DisPCFig5c}.

The same partition of $\partial H$ into the union of $\mathcal{A}$, $F_A$ and $F_B$, can be used to obtain an R-R diagram of $\alpha$ and $\beta$ on $\partial H$ as was used in the previous case. The R-R diagram will differ of course, because now $\alpha$ has more than one connection in each of $F_A$, $F_B$. We can determine what these connections of $\alpha$ can be by looking at the cyclic word which $\alpha$ represents in $\pi_1(H)$.

Theorem \ref{recognizing primitives} below, which is the main result of \cite{CMZ81}, shows that, if $\alpha$ is primitive in $H$, then the cyclic word which $\alpha$ represents in $\pi_1(H)$ must have a particular form, and this provides what we need. To use Theorem \ref{recognizing primitives}, let $A$ and $B$ be generators of $\pi_1(H)$ chosen so that $A$ and $B$ are represented by simple closed curves in $H$ dual to $D_A$ and $D_B$ respectively. Then, since $\alpha$ is primitive in $H$, and $G_{\alpha,\beta}$ has the form of Figure \ref{DisPCFig5c} with $|\alpha \cap \partial D_B| > 1$, Theorem \ref{recognizing primitives} implies $\alpha$ represents a cyclic word in $\pi_1(H)$ of the form $w = A^{m_1}B \dots A^{m_j}B$, with $\{m_1, \dots ,m_j\}$ = $\{e, e+1\}$ and $j > 1$. In addition, since $j > 1$ and $\alpha$ is a primitive rather than a proper power of a primitive in $H$, both $e$ and $e+1$ must appear as exponents of $A$ in $W$.

It follows that the A-handle of an R-R diagram $\mathcal{D}$ describing the embedding of $\alpha$ and $\beta$ in $\partial H$ must have exactly two nonparallel types of connections bearing labels $p$ and $p+\epsilon$, where $\epsilon = \pm 1$, and $\{p, p+ \epsilon\}$ = $\{e,e+1\}$.

Next, let $a = |A^pB|$ and $b = |A^{p+\epsilon}B|$ be the number of subwords of the form $A^pB$ and $A^{p+\epsilon}B$ respectively in $w$. Then $a + b = |\alpha \cap \partial D_B| > 1$. And, since $c \geq a + b$ in Figure \ref{DisPCFig5c}, the set of exponents $\{m_1, \dots ,m_j\}$ of $A$ in $w$ must satisfy $\{m_1, \dots ,m_j\}$ = $\{e, e+1\}$ with $e > 1$. So $\min\{p, p+\epsilon\} > 1$ in $\mathcal{D}$. Finally, since $\{A^pB, A^{p+\epsilon}B\}$ is a set of free generators of $\pi_1(H)$, $\gcd(a,b) = 1$.

Then, putting this all together, it follows readily that the pair $(\alpha, \beta)$ has an R-R diagram with the form of Figure~\ref{DisPCFig3a}.

\end{proof}

\begin{lem}\label{Case 3 Heegaard diagram}
Suppose $\alpha$ and $\beta$ are disjoint nonparallel primitive simple closed curves in the boundary of a genus two handlebody $H$, and $\{D_A, D_B\}$ is a complete set of cutting disks of $H$  such that $|\beta \cap \partial D_A| = 0$, $|\beta \cap \partial D_B| = 1$, and $|\alpha \cap \partial D_B| > 1$, where $|\alpha \cap \partial D_B|$ is minimal subject to $|\beta \cap \partial D_B| = 1$.

Then the graph $G_{\alpha, \beta}$ underlying the Heegaard diagram $HD_{\alpha, \beta}$ of $\alpha$ and $\beta$ with respect to $\{D_A, D_B\}$ has the form of Figure \emph{\ref{DisPCFig5c}} with $c \geq a + b > 1$.
\end{lem}

\begin{proof}
Since $\alpha$ is primitive in $H$, there is a cutting disk $D$ of $H$ disjoint from $\alpha$. This implies the subgraph $G_\alpha$ of $G_{\alpha, \beta}$ is either not connected, or has a cut vertex. We will show that the only way either of these alternatives can hold is if $G_{\alpha, \beta}$ has the form of Figure \ref{DisPCFig5c}.

First, note that if there are no edges of $HD_\alpha$ connecting $D_A^+$ to either $D_B^+$ or $D_B^-$, then $|\alpha \cap \partial D_B| \leq 1$, contrary to hypothesis. So there must be edges of $G_\alpha$ connecting $A^+$ to either $B^+$ or $B^-$. However, if there are edges of $G_\alpha$ connecting $A^+$ to both $B^+$ and $B^-$, then $G_\alpha$ is connected and has no cut vertex. So, up to swapping $A^+$ and $A^-$ or $B^+$ and $B^-$, we may assume $A^+$ is connected to $B^-$, and $A^+$ is not connected to $B^+$ in $G_\alpha$.

Continuing, if edges of $HD_\alpha$ only connect $D_A^+$ to $D_B^-$, and $D_A^-$ to $D_B^+$, then the bandsum of $D_A$ and $D_B$ along one of these edges is a cutting disk $D_B'$ of $H$ such that $\{D_A, D_B'\}$ is a complete set of cutting disks of $H$, $|\beta \cap \partial D_B'| = 1$, and $|\alpha \cap \partial D_B'| < |\alpha \cap \partial D_B|$. This contradicts the assumed minimality of $|\alpha \cap \partial D_B|$.

It follows that there are edges of $G_\alpha$ connecting $A^+$ to $A^-$. And then, since $G_\alpha$ must have a cut vertex, no edges of $G_\alpha$ connect $B^+$ to $B^-$. 

Finally, if $c < a + b$ in Figure \ref{DisPCFig5c}, then, as before, the bandsum of $D_A$ and $D_B$ along an edge of $HD_\alpha$ connecting $D_A^+$ to $D_B^-$ is a cutting disk $D_B'$ of $H$ such that $\{D_A, D_B'\}$ is a complete set of cutting disks of $H$, $|\beta \cap \partial D_B'| = 1$, and $|\alpha \cap \partial D_B'| < |\alpha \cap \partial D_B|$. This again contradicts the assumed minimality of $|\alpha \cap \partial D_B|$. It follows that $G_{\alpha,\beta}$ has the form shown in Figure \ref{DisPCFig5c}.
\end{proof}

\subsection{Recognizing primitives}

The following result of Cohen, Metzler, and Zimmerman makes it possible to determine if a cyclically reduced word in a free group of rank two is primitive. 

\begin{thm}[CMZ81]
\label{recognizing primitives}
Suppose a cyclic conjugate of 
\[w = A^{m_1}B^{n_1} \dots A^{m_j}B^{n_j}\]
 is a member of a basis of $F(A,B)$, where $j \geq 1$ and each indicated exponent is nonzero. Then, after perhaps replacing $A$ by $A^{-1}$ or $B$ by $B^{-1}$, there exists $e > 0$ such that:
\[
m_1 = \dots = m_j = 1,
\quad
\text{and}
\quad
\{n_1, \dots ,n_j\} = \{e, e+1\},
\]
or
\[
\{m_1, \dots ,m_j\} = \{e, e+1\},
\quad
\text{and}
\quad
n_1 = \dots = n_j = 1.
\]
\end{thm}

Note that if $w$ in $F(A,B)$ has the form $w = AB^{n_1} \dots AB^{n_j}$, say, with $j \geq 1$ and $\{n_1, \dots ,n_j\} = \{e, e+1\}$, then the automorphism $A \mapsto AB^{-e}$ of $F(A,B)$ reduces the length of $w$, so repeated applications of such automorphisms can be used to determine if a given word $w$ in $F(A,B)$ is a primitive.

\begin{figure}[htbp]
\centering
\includegraphics[width = 0.55\textwidth]{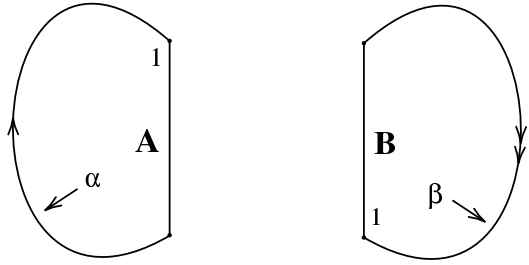}
\caption{If $H$ is a genus two handlebody, $(\alpha, \beta)$ is a pair of disjoint nonparallel primitives in $\partial H$, and $\{D_A, D_B\}$ is a complete set of cutting disks of $H$ such that $|\beta \cap \partial D_A| = 0$, $|\beta \cap \partial D_B| = 1$, and $|\alpha \cap \partial D_B| = 0$, then $\alpha$ and $\beta$ have an R-R diagram with the form of this figure. Here $(\alpha, \beta)$ represents $(A,B)$ in $\pi_1(H)$.}
\label{DisPCFig1a}
\end{figure}

\begin{figure}[htbp]
\centering
\includegraphics[width = 0.7\textwidth]{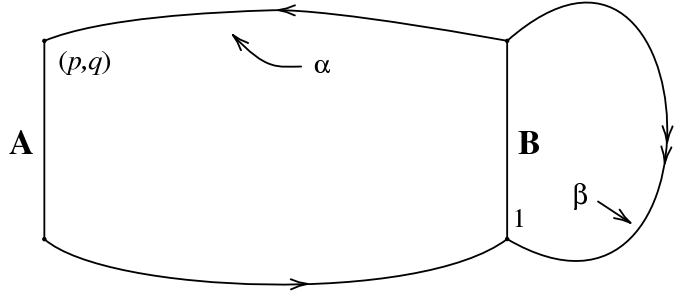}
\caption{If $H$ is a genus two handlebody, $(\alpha, \beta)$ is a pair of disjoint nonparallel primitives in $\partial H$, and $\{D_A, D_B\}$ is a complete set of cutting disks of $H$ such that $|\beta \cap \partial D_A| = 0$, $|\beta \cap \partial D_B| = 1$, and $|\alpha \cap \partial D_B| = 1$, then $\alpha$ and $\beta$ have an R-R diagram with the form shown in this figure. Here the parameters $p$ and $q$ are intersection numbers of the connection $\alpha \cap F_A$ with $\partial D_A$ and a simple closed longitudinal curve $l$ on $F_A$ with $|l \cap \partial D_A| = 1$. Then $p \in \mathbb{Z}$, $\gcd(p,q) = 1$, and $(\alpha, \beta)$ represents $(A^p,B)$ in $\pi_1(H)$.}
\label{DisPCFig2a}
\end{figure}

\begin{figure}[htbp]
\centering
\includegraphics[width = 0.8\textwidth]{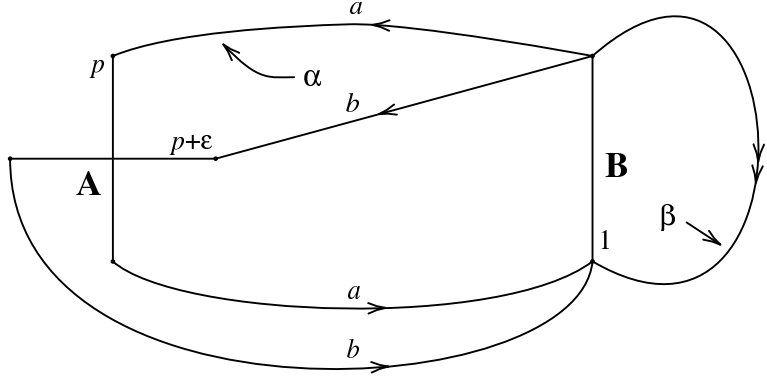}
\caption{If $H$ is a genus two handlebody, $(\alpha, \beta)$ is a pair of disjoint nonparallel primitives in $\partial H$, and $\{D_A, D_B\}$ is a complete set of cutting disks of $H$ with $|\beta \cap \partial D_A| = 0$, $|\beta \cap \partial D_B| = 1$, and $|\alpha \cap \partial D_B| = s > 1$, with $s$ minimal subject to $|\beta \cap \partial D_B| = 1$, then $\alpha$ and $\beta$ have an R-R diagram with the form shown in this figure with parameters which satisfy $\gcd(a,b) = 1$, $a + b > 1$, $\epsilon = \pm1$, and $\text{min}\{p, p+\epsilon\} > 1$.}
\label{DisPCFig3a}
\end{figure}

\begin{figure}[htbp]
\centering
\includegraphics[width = 0.35\textwidth]{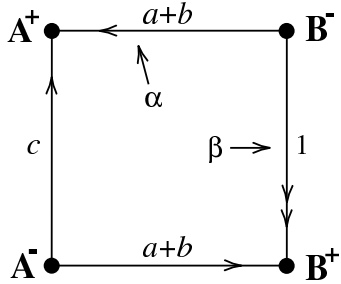}
\caption{Lemma \ref{Case 3 Heegaard diagram} shows that if $|\alpha \cap \partial D_B| > 1$ in Theorem \ref{3 case theorem}, then, up to swapping $A^+$ and $A^-$ or $B^+$ and $B^-$, the graph $G_{\alpha, \beta}$ of the  Heegaard diagram $HD_{\alpha, \beta}$ has the form shown here with $c \geq a + b$ and $a + b = |\alpha \cap \partial D_B|$.}
\label{DisPCFig5c}
\end{figure}

\begin{figure}[htbp]
\centering
\includegraphics[width = 0.8\textwidth]{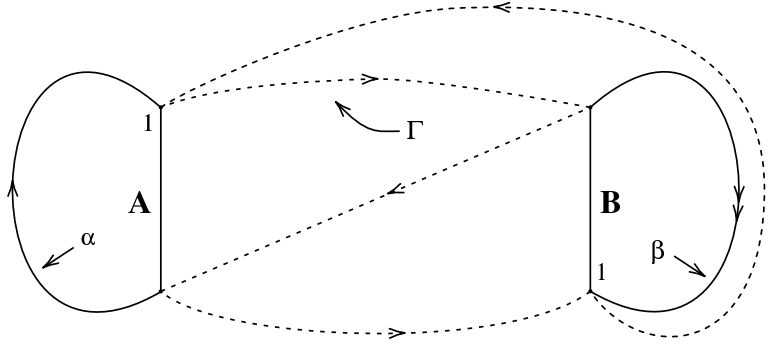}
\caption{This figure shows the pair of disjoint primitives $(\alpha, \beta)$ of Figure \ref{DisPCFig1a} separated by a simple closed curve $\Gamma$ such that $\Gamma$ represents $ABA^{-1}B^{-1}$ in $\pi_1(H)$. (The existence of $\Gamma$ shows that in addition to being separated and a Type I pair, $(\alpha, \beta)$ is also a Type II pair.)}
\label{DisPCFig1b}
\end{figure}

\begin{figure}[htbp]
\centering
\includegraphics[width = 0.85\textwidth]{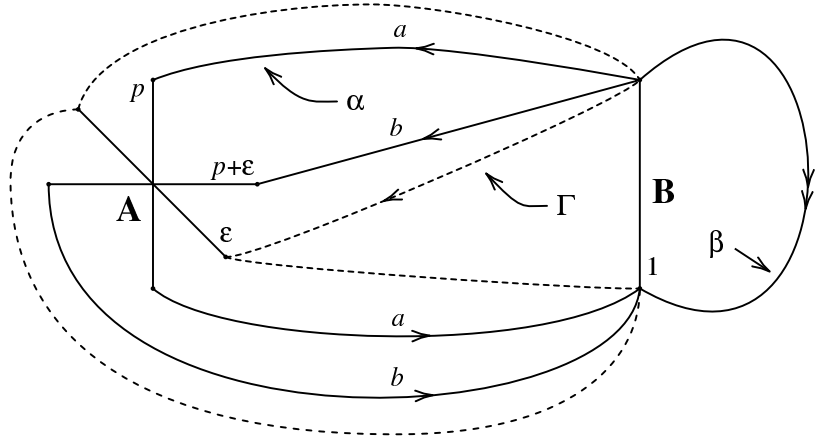}
\caption{This figure and Lemma \ref{Commutator invariance} show that, if $(\alpha, \beta)$ is a pair of disjoint primitives in $\partial H$ with an R-R diagram of the form shown in Figure \ref{DisPCFig3a}, then there exists a curve $\Gamma$ in $\partial H$, separating $\alpha$ and $\beta$, and a once-punctured torus $F$, such that $H$ is homeomorphic to $F \boldsymbol{\times} I$, under a homeomorphism which takes $\partial F \boldsymbol{\times} I$ to a regular neighborhood of $\Gamma$ in $\partial H$. This follows from Lemma \ref{Commutator invariance}, since $\Gamma$ represents $A^{\epsilon}B^{-1}A^{-\epsilon}B$ in $\pi_1(H)$ with $\epsilon = \pm 1$.}
\label{DisPCFig3b}
\end{figure}

\begin{figure}[htbp]
\centering
\includegraphics[width = 0.8\textwidth]{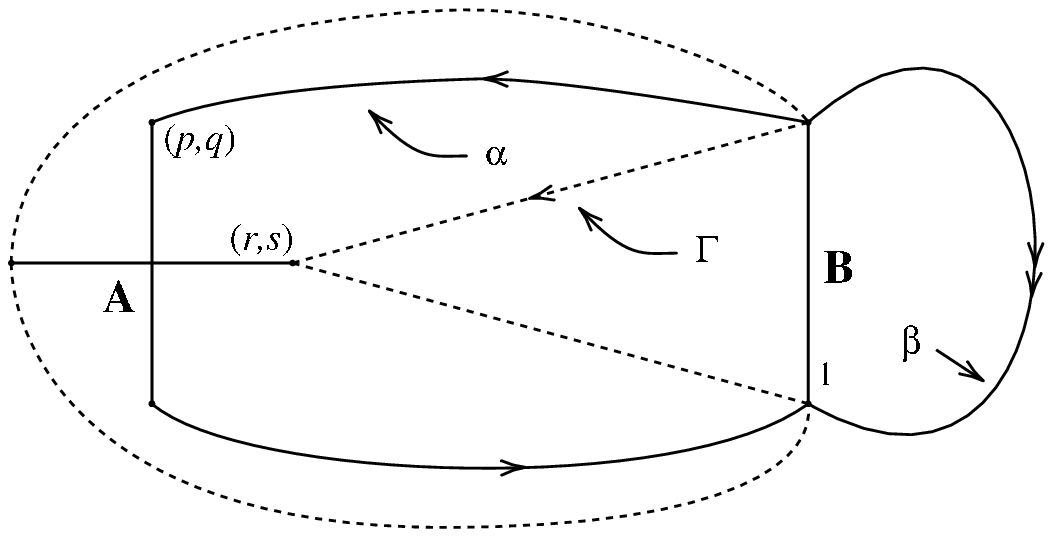}
\caption{This R-R diagram shows that if $\alpha$ and $\beta$ have an R-R diagram with the form of Figure \ref{DisPCFig2a}, and $\Gamma$ is a simple closed curve in $\partial H$, separating $\alpha$ and $\beta$, such that $\Gamma$ represents a cyclic word in $\pi_1(H)$ with no more than four syllables, then $\alpha$, $\beta$, and $\Gamma$ have an R-R diagram with the form of this figure. Here the pairs of parameters $(p,q)$ and $(r,s)$ are intersection numbers of connections on the A-handle of this R-R diagram 
with $\partial D_A$ and a simple closed longitudinal curve $l$ on the A-handle. (By twisting $l$ around $\partial D_A$ when $p \neq 0$, we may assume $|r| < |p|$.) Note $|l \cap \partial D_A| = 1$ if and only if $ps - rq = \pm 1$. Then, by Lemma \ref{Commutator invariance}, $(\alpha, \beta)$ is a Type II pair if and only if $r = \pm 1$, and this occurs if and only if $q = \pm(ps \mp 1)$.
}
\label{DisPCFig2b}
\end{figure}

\subsection{Ends of (twisted) products}

\begin{defn}
Suppose $F$ is a once-punctured torus. Then $F \boldsymbol{\times} I$ is a genus two handlebody $H$, and the surfaces $F \boldsymbol{\times} 0$ and $F \boldsymbol{\times} 1$ in $\partial H$ are the \emph{ends} of the product $F \boldsymbol{\times} I$. Suppose $\delta$ is a nonseparating simple closed curve lying in $F \boldsymbol{\times} 0$ or $F \boldsymbol{\times} 1$. If $\delta$ is pushed into the interior of $F \boldsymbol{\times} I$ by an isotopy and Dehn surgery is performed on $\delta$, the result is another genus two handlebody $H'$. Then $H'$ is a \emph{twisted product} $F \widetilde{\boldsymbol{\times}} I$, and the two surfaces $F \widetilde{\boldsymbol{\times}} 0$ and $F \widetilde{\boldsymbol{\times}} 1$ in $\partial H'$ are the \emph{ends} of the twisted product $F \widetilde{\boldsymbol{\times}} I$.
\end{defn}

\subsection{Type I and Type II pairs of disjoint primitives}

\begin{defn}
Suppose $(\alpha, \beta)$ is a pair of disjoint nonparallel primitive simple closed curves in the boundary of a genus two handlebody $H$. The pair $(\alpha, \beta)$ is a \emph{Type} I pair if there is a cutting disk $D$ of $H$ such that $|\alpha \cap \partial D| = 1$ and $|\beta \cap \partial D| = 1$. The pair $(\alpha, \beta)$ is a \emph{Type} II pair if there is a once-punctured torus $F$ and a homeomorphism $h \colon H \to F \boldsymbol{\times} I$ such that $h(\alpha) \subset F \boldsymbol{\times} 1$ and $h(\beta) \subset F \boldsymbol{\times} 0$. (Somewhat loosely, Type I pairs are \emph{twisted pairs}, while Type II pairs are \emph{untwisted pairs}.)
\end{defn}

The remaining results of this section show that pairs of disjoint nonparallel primitives on the boundary of a genus two handlebody either lie on disjoint ends of an ordinary product $F \boldsymbol{\times} I$, or on disjoint ends of a twisted product $F \widetilde{\boldsymbol{\times}} I$, where $F$ is a once-punctured torus.

We begin with the following lemma which characterizes $(\alpha, \beta)$ pairs which lie on disjoint ends of a product, $F \boldsymbol{\times} I$, where $F$ is a once-punctured torus.

\begin{lem} \label{Commutator invariance}
Suppose $H$ is a genus two handlebody, $\alpha$ and $\beta$ are two disjoint nonparallel simple closed curves in $\partial H$, each of which is primitive in $H$, and  $A$ and $B$ are a pair of free generators of $\pi_1(H)$. Then there is a once-punctured torus $F$ and a homeomorphism $h \colon H \to F \boldsymbol{\times} I$ such that $h(\alpha) \subset F \boldsymbol{\times} 1$ and $h(\beta) \subset F \boldsymbol{\times} 0$ if and only if there exists a simple closed curve $\Gamma$ in $\partial H$ separating $\alpha$ and $\beta$ such that $\Gamma$ represents the cyclic word $ABA^{-1}B^{-1}$ or its inverse in $\pi_1(H)$.
\end{lem}

\begin{proof}
The proof follows directly from the well-known fact, see Proposition 5.1 of \cite{LS77}, that any automorphism of the free group of rank two $F(A,B)$ carries the cyclic word represented by the commutator $ABA^{-1}B^{-1}$ onto itself or its inverse. 

Suppose there is a once-punctured torus $F$ and a homeomorphism $h \colon H \to F \boldsymbol{\times} I$ such that $h(\alpha) \subset F \boldsymbol{\times} 1$ and $h(\beta) \subset F \boldsymbol{\times} 0$. Let $X$ and $Y$ be a pair of free generators of $\pi_1(F)$. Then the simple closed curve $\Gamma = h^{-1}(\partial F \boldsymbol{\times} \frac{1}{2})$ in $\partial H$ separates $\alpha$ and $\beta$ and $\Gamma$ represents a cyclic word in $\pi_1(H)$ equal to the commutator $h^{-1}(XYX^{-1}Y^{-1})$ or its inverse in $\pi_1(H)$.

Conversely, suppose there exists a simple closed curve $\Gamma$ in $\partial H$ separating $\alpha$ and $\beta$, and a pair of free generators $A$ and $B$ of $\pi_1(H)$ such that $\Gamma$ represents the cyclic word $ABA^{-1}B^{-1}$ or its inverse in $\pi_1(H)$. Then there exists a complete set of cutting disks $\{D_A, D_B\}$ of $H$ such that $|\partial D_A \cap \Gamma| = |\partial D_B \cap \Gamma| = 2$. It follows readily that there exists a once-punctured torus $F$, together with a pair of nonparallel connections $\delta_A$, $\delta_B$ in $F$, and a homeomorphism $h \colon H \to F \boldsymbol{\times} I$ such that $D_A = h^{-1}(\delta_A \boldsymbol{\times} I)$, $D_B = h^{-1}(\delta_B \boldsymbol{\times} I)$, and $\Gamma = h^{-1}(\partial F \boldsymbol{\times} \frac{1}{2})$. So, in particular, either $h(\alpha) \subset F \boldsymbol{\times} 1$ and $h(\beta) \subset F \boldsymbol{\times} 0$, or $h(\beta) \subset F \boldsymbol{\times} 1$ and $h(\alpha) \subset F \boldsymbol{\times} 0$.
\end{proof}

\begin{cor}
If $H$ is a genus two handlebody and $(\alpha, \beta)$ is a pair of disjoint primitives in $\partial H$ with an R-R diagram of the form of Figure \emph{\ref{DisPCFig3a}}, then $\alpha$ and $\beta$ lie on disjoint ends of a product $F \boldsymbol{\times} I$, where $F$ is a once-punctured torus.
\end{cor}

\begin{proof}
Figure \ref{DisPCFig3b} shows that, if $\alpha$ and $\beta$ have an R-R diagram with the form of Figure \ref{DisPCFig3a}, then there exists a simple closed $\Gamma$ in $\partial H$ separating $\alpha$ and $\beta$ such that $\Gamma$ represents $A^{\epsilon}B^{-1}A^{-\epsilon}B$ in $\pi_1(H)$ with $\epsilon = \pm1$. Then Lemma \ref{Commutator invariance} implies the claim.
\end{proof}

\begin{lem}
Suppose $H$ is a genus two handlebody and $(\alpha, \beta)$ is a pair of disjoint primitives in $\partial H$ with an R-R diagram of the form of Figure \emph{\ref{DisPCFig2a}}. 
\begin{enumerate}
\item If $r = 0$, $\alpha$ and $\beta$ are separated and also have an R-R diagram with the form of Figure \emph{\ref{DisPCFig1a}}.
\item If $|r| = 1$, then $\alpha$ and $\beta$ lie on disjoint ends of a product $F \boldsymbol{\times} I$, where $F$ is a once-punctured torus.
\item If $|r| > 1$, then $\alpha$ and $\beta$ lie on disjoint ends of a twisted product $F \widetilde{\boldsymbol{\times}} I$, where $F$ is a once-punctured torus.
\end{enumerate}
\end{lem}

\begin{proof}
Figure \ref{DisPCFig2b} shows that, if $\alpha$ and $\beta$ have an R-R diagram with the form of Figure \ref{DisPCFig3a}, then there exists a simple closed $\Gamma$ in $\partial H$ separating $\alpha$ and $\beta$ such that $\Gamma$ represents $A^{r}B^{-1}A^{-r}B$ in $\pi_1(H)$. If $r = 0$, $\Gamma$ bounds a separating disk in $H$. So $\alpha$ and $\beta$ are separated. Otherwise, if $|r| = 1$, Lemma \ref{Commutator invariance} shows $\alpha$ and $\beta$ lie on disjoint ends of a product $F \boldsymbol{\times} I$. 

This leaves the case $|r| > 1$. In this case, there is a nonseparating simple closed curve $\lambda$ in $F_A$ such that $|\Gamma \cap \lambda| = 2$. Then there is a Dehn surgery on a core curve of the A-handle of $H$ which turns $H$ into another genus two handlebody $H'$ in which $\lambda$ bounds a cutting disk and, by Lemma \ref{Commutator invariance}, $\alpha$ and $\beta$ lie on disjoint ends of a product $F \boldsymbol{\times} I$. From this it is easy to see that $\alpha$ and $\beta$ lie on disjoint ends of a twisted product $F \widetilde{\boldsymbol{\times}} I$.
\end{proof}

This next result characterizes the cyclic words in $\pi_1(H)$ which are represented by the separating simple closed curves $\Gamma$ in Figures \ref{DisPCFig3b} and \ref{DisPCFig2b}. 

\begin{prop} \label{Disjoint primitives sit on a product or twisted product}
Suppose $H$ is a genus two handlebody with a pair of disjoint nonparallel simple closed curves $\alpha$ and $\beta$ in $\partial H$ such that both $\alpha$ and $\beta$ are primitive in $H$. Then there is a simple closed curve $\Gamma$ in $\partial H$ separating $\alpha$ and $\beta$, and a complete set of cutting disks $\{D_A, D_B\}$ of $H$, such that either $|\Gamma \cap \partial D_A| = 2$, or $|\Gamma \cap \partial D_B| = 2$. In particular, up to replacing $A$ with $A^{-1}$, $B$ with $B^{-1}$, or perhaps exchanging $A$ and $B$, there is an integer $n$ such that $\Gamma$ represents the cyclic word $A^nBA^{-n}B^{-1}$ in $\pi_1(H)$.
\end{prop}

\begin{proof}
Examination of the curve $\Gamma$ in Figures \ref{DisPCFig1b}, \ref{DisPCFig3b}, and \ref{DisPCFig2b}, shows that, in each case, $\Gamma$ separates $\alpha$ and $\beta$ and $|\Gamma \cap \partial D_B| = 2$. It follows that $\Gamma$ represents a cyclic word in $\pi_1(H)$ of the claimed form.
\end{proof}

Finally, the following theorem restates the classification in terms of Type I and Type II pairs.

\begin{thm}\label{Type I and Type II Thm}
If $(\alpha, \beta)$ is a pair of disjoint nonparallel primitive simple closed curves in the boundary of a genus two handlebody $H$, then $(\alpha, \beta)$ is either a Type~\emph{I} or Type~\emph{II} pair. In particular:
\begin{enumerate}
\item $(\alpha, \beta)$ is a Type~\emph{I} pair if $\alpha$ and $\beta$ have an R-R diagram with the form of Figure \emph{\ref{DisPCFig2a}}. 

\item $(\alpha, \beta)$ is a Type~\emph{II} pair if $\alpha$ and $\beta$ have an R-R diagram with the form of Figure \emph{\ref{DisPCFig3a}}. 

\item $(\alpha, \beta)$ is both a Type~\emph{I} pair and a Type~\emph{II} pair if and only if $\alpha$ and $\beta$ are separated in $H$, or $\alpha$ and $\beta$ have an R-R diagram with the form of Figure \emph{\ref{DisPCFig2a}} in which $\alpha$ wraps around the A-handle of the R-R diagram $p$ times longitudinally and $q$ times meridionally with $q = ps \pm 1$ for some integer s.  \emph{(}See Figure~\emph{\ref{DisPCFig2b}}.\emph{)}
\end{enumerate}
\end{thm}

\begin{proof}
This follows from the other results of this section. Details omitted.
\end{proof}

\section{Pairs in which $\beta$ is a proper power and $\alpha$ is a primitive or proper power}

A nonseparating simple closed curve $\beta$ in the boundary of a genus two handlebody $H$ is a \emph{proper power} if $\beta$ is disjoint from an essential separating disk in $H$, $\beta$ does not bound a disk in $H$, and $\beta$ is not primitive in $H$.
With the classification of $\alpha$, $\beta$ pairs in which both $\alpha$ and $\beta$ are primitives finished, it seems natural to generalize slightly to the situation in which one or both of $\alpha$, $\beta$ are proper powers of primitives.

Here, as promised in the abstract, the situation is simpler and completely described by the following theorem. 
 
\begin{thm} \label{The case beta is a proper power}
Suppose $H$ is a genus two handlebody, and $\alpha$ and $\beta$ are two disjoint nonparallel simple closed curves in $\partial H$ such that $\beta$ is a proper power in $H$ and $\alpha$ is primitive or a proper power in $H$. Then either $\alpha$ and $\beta$ are separated in $H$, or $\alpha$ and $\beta$ bound a nonseparating annulus in $H$.
\end{thm}

\begin{proof}
Suppose $\{D_A, D_B\}$ is a complete set of cutting disks of $H$ with the property that $|\beta \cap (\partial D_A \cup \partial D_B)|$ is minimal, and also $|\alpha \cap (\partial D_A \cup \partial D_B)|$ is as small as possible among the complete sets of cutting disks of $H$ minimizing $|\beta \cap (\partial D_A \cup \partial D_B)|$.

Then, since $\beta$ is a proper power in $H$, one of $D_A$, $D_B$, say $D_A$, is disjoint from $\beta$. And then $|\beta \cap \partial D_B| = s$ with $s > 1$. 

\begin{claim}\label{Claim 1}
$\alpha$ intersects only one of $\partial D_A$, $\partial D_B$.
\end{claim}

\begin{proof}[Proof of Claim \emph{\ref{Claim 1}}]
We use the notation of Subsection \ref{H-diagrams and underlying graphs}.
Then, since $\alpha$ is disjoint from a disk in $H$, $G_\alpha$ is either not connected, or it has a cut vertex. And, of course, the graph $G_\beta$ is not connected, since all $s$ of its edges connect $D_B^+$ to $D_B^-$. 

Now the proof of Claim \ref{Claim 1} breaks into two cases depending upon whether there are nonparallel edges in $G_\beta$.

\noindent \textbf{Case 1:} There are edges of $G_\beta$ which are not parallel. \hfill 

In this case, the vertices $D_A^+$ and $D_A^-$ of $G_\beta$ lie in different faces of $G_\beta$. Since $\alpha$ and $\beta$ are disjoint, this implies there are no edges of $G_\alpha$ connecting $D_A^+$ to $D_A^-$ in $G_\alpha$. It follows that, up to exchanging $D_A^+$ and $D_A^-$ or $D_B^+$ and $D_B^-$, $G_\alpha$ has the form of Figure \ref{DisPCFig5a}.

If $a = 0$ in Figure \ref{DisPCFig5a} the claim holds. So suppose $a > 0$ in Figure \ref{DisPCFig5a}. Then the bandsum of $D_A$ and $D_B$, along a subarc of $\alpha$ representing one of the $a$ edges of $G_\alpha$ connecting $D_A^+$ to $D_B^-$, is a disk $D_B'$ of $H$ such that $\{D_A, D_B'\}$ is a complete set of cutting disks of $H$, $|\beta \cap \partial D_B'| = s$, and $|\alpha \cap \partial D_B'| < |\alpha \cap \partial D_B|$, contrary to hypothesis. It follows that $a = 0$ in Figure \ref{DisPCFig5a}, and the claim holds in this case.

\noindent \textbf{Case 2:} Any two edges of $G_\beta$ are parallel. \hfill

In this case, since any two edges of $G_\beta$ are parallel, there is a once-punctured torus $F$ in $\partial H$ which contains $\partial D_B$ and $\beta$. If $\alpha$ and $\partial F$ are disjoint, then $\alpha$ and $\partial D_B$ are disjoint and the claim holds. So suppose the set of connections $F \cap \alpha$ is nonempty, and $\delta$ is a connection in $F \cap \alpha$. Then, since $\alpha$ and $\beta$ are disjoint, $|\delta \cap \partial D_B| = s > 1$. This implies the graph $G_\alpha$ also has the form of Figure \ref{DisPCFig5a} in this case. Then the argument of Case 1 shows the claim also holds here. This finishes the proof of Claim \ref{Claim 1}.
\end{proof}

To finish the proof of Theorem \ref{The case beta is a proper power}, observe that if $\alpha$ only intersects $\partial D_A$, then clearly $\alpha$ and $\beta$ are separated in $H$. On the other hand, if $\alpha$ only intersects $\partial D_B$, then both $\alpha$ and $\beta$ are disjoint from $D_A$, and so $\alpha$ and $\beta$ bound an annulus $\mathcal{A}$ in the solid torus obtained by cutting $H$ open along $D_A$. Finally, $\mathcal{A}$ must be nonseparating in $H$ since $\alpha$ and $\beta$ are not parallel in $\partial H$.
\end{proof}

\begin{figure}[htbp]
\centering
\includegraphics[width = 0.35\textwidth]{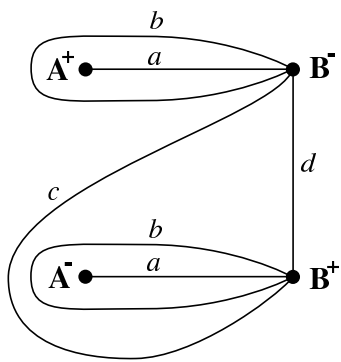}
\caption{The graph $G_\alpha$ of the proof of Theorem \ref{The case beta is a proper power}.}
\label{DisPCFig5a}
\end{figure}

\end{document}